\documentclass{amsart}
\usepackage{latexsym,amsmath,amssymb}
\newtheorem{thm}{Theorem}[section]
\newtheorem{cor}[thm]{Corollary}
\newtheorem{exam}[thm]{Example}
\newtheorem{lem}[thm]{Lemma}

\theoremstyle{definition}
\theoremstyle{remark}
\newtheorem{rem}[thm]{Remark}
\numberwithin{equation}{section}

\begin{document}

\title[]
{Composition operators on some semi-Hilbert spaces}

\author{\sc\bf Y. Estaremi and M. S. Al Ghafri}
\address{\sc Y. Estaremi}
\email{y.estaremi@gu.ac.ir}
\address{Department of Mathematics, Faculty of Sciences, Golestan University, Gorgan, Iran.}
\address{\sc M. S. Al Ghafri}
\address{Department of Mathematics, University of Technology and Applied Sciences, Rustaq,   Sultanate of Oman.}
\email{mohammed.alghafri@utas.edu.om}
\address{}
\thanks{}

\thanks{}



\date{}

\dedicatory{}

\commby{}

\begin{abstract}
This paper investigates composition operators and weighted composition operators on semi-Hilbert spaces induced by positive multiplication operators on \( L^2(\mu) \). Within the framework of \( A \)-adjoint operators, we characterize conditions under which a composition operator \( C_\varphi \) is \( A \)-selfadjoint, \( A \)-normal, \( A \)-quasinormal, or \( A \)-unitary, where \( A = M_u \) denotes a positive multiplication operator. Additionally, we provide necessary and sufficient conditions for \( C_\varphi \) to be an \( A \)-isometry or an \( A \)-partial isometry. Analogous results are also obtained when \( A \) is a positive composition operator. Several examples are presented to illustrate the applicability of the main results.
\end{abstract}

\maketitle
\section{Introduction and Preliminaries}

The study of composition operators \( C_{\varphi} \) and multiplication operators  \( M_u \) in function spaces, has led to significant advances in characterizing their adjoints, isometric properties, and spectral behavior \cite{Hedenmalm2000,DurenSchuster2004,Zhu2005,tak}. In particular, the self-adjointness of weighted composition operators \( W = M_u C_{\varphi} \) in \( L^2(\mu) \) and their unitary equivalences have been explored in recent literature \cite{Cowen1995}.
A fundamental aspect of operator theory is the study of adjoint operators. The concept of \( A \)-adjoint operators, which generalizes classical adjoints, is crucial in defining an alternative inner product structure. Let \( \mathcal{H} \) be a complex Hilbert space with inner product $\langle \cdot , \cdot \rangle$ and norm $\|\cdot\|$. Let $\mathcal{B}(\mathcal{H})$ denotes   the Banach algebra of all bounded linear operators on $\mathcal{H}$. $\mathcal{B}(\mathcal{H})^{+}$ is the cone of positive (semi-definite) operators, i.e.,

\[
\mathcal{B}(\mathcal{H})^{+} = \{A \in \mathcal{B}(\mathcal{H}) : \langle A\xi , \xi \rangle \geq 0, \, \forall \xi \in \mathcal{H} \}.
\]

For any operator $T$ in $\mathcal{B}(\mathcal{H})$, its range and null space are denoted by $\mathcal{R}(T)$ and $\mathcal{N}(T)$, respectively, while its adjoint is represented as $T^*$. If $\mathcal{M}$ is a closed subspace of $\mathcal{H}$, the orthogonal projection onto $\mathcal{M}$ is written as $P_{\mathcal{M}}$. A subspace $\mathcal{M}$ is said to be invariant under $T$ if $T\mathcal{M} \subset \mathcal{M}$. Moreover, $\mathcal{M}$ is called a reducing subspace for $T$ when both $\mathcal{M}$ and its orthogonal complement $\mathcal{M}^{\perp}$ remain invariant under $T$. The set of complex numbers is denoted by $\mathbb{C}$, and the complex conjugate of a number $\lambda$ is expressed as $\overline{\lambda}$. Additionally, the closure of $\mathcal{R}(T)$ is represented as $\overline{\mathcal{R}(T)}$.

Let \( A \in \mathcal{B}(\mathcal{H})^{+} \) be a positive operator. The semi-inner product induced by \( A \) is given by
\[
\langle \xi, \eta \rangle_A = \langle A\xi, \eta \rangle, \quad \forall \xi, \eta \in \mathcal{H}.
\]
Then $< ., . >_A$ is a semi-inner product on $\mathcal{H}$. It is clear that
\[
\|\xi\|_A = \langle \xi, \xi \rangle_A^{\frac{1}{2}} = \langle A\xi, \xi \rangle^{\frac{1}{2}} = \langle A^{\frac{1}{2}} \xi, A^{\frac{1}{2}} \xi \rangle^{\frac{1}{2}} = \|A^{\frac{1}{2}} \xi\|.
\]
It's clear that $\|\xi\|_A = 0$ if and only if $\xi \in \mathcal{N}(A)$. Then $\|\cdot\|_A$ is a norm if and only if $A$ is an injective operator. Moreover, $\langle ., . \rangle_A$ defines a seminorm on a certain subset of $\mathcal{B}(\mathcal{H})$ that contains all $T \in \mathcal{B}(\mathcal{H})$ for which there exists a constant $c > 0$ such that
\[
\|T\xi\|_A \leq \|\xi\|_A, \quad \text{for all } \xi \in \mathcal{H}.
\]
For these operators we have
\[
\|T\|_A = \sup_{\xi \in \mathcal{R}(A), \, \xi \neq 0} \frac{\|T\xi\|_A}{\|\xi\|_A} < \infty.
\]

For $\mathcal{L} \subseteq \mathcal{H}$ we denote by
\[
\mathcal{L}^{\perp_A} = \{ \xi \in \mathcal{H} : \langle \xi, \eta \rangle_A = 0, \, \forall \eta \in \mathcal{L} \}.
\]
It is easily seen that $\mathcal{L}^{\perp_A} = (A\mathcal{L})^\perp = A^{-1}(\mathcal{L}^\perp)$. Moreover, since $A(A^{-1}(\mathcal{L})) = \mathcal{L} \cap \mathcal{R}(A)$, then we have
\[
(\mathcal{L}^{\perp_A})^{\perp_A} = \mathcal{L}^{\perp_A} \cap \mathcal{R}(A)^{\perp}.
\]
Within this framework, the concept of an \( A \)-adjoint becomes central. An operator \( T \in \mathcal{B}(\mathcal{H}) \) is said to admit an \( A \)-adjoint if there exists $W \in \mathcal{B}(\mathcal{H})$ such that 
\[
\langle W\xi, \eta \rangle_A = \langle \xi, T\eta \rangle_A
\]
for all $\xi, \eta \in \mathcal{H}$. Hence $T$ has an $A$-adjoint if and only if there exists $W \in \mathcal{B}(\mathcal{H})$ such that $AW = T^* A$. Therefore, $T$ has an $A$-adjoint if and only if the equation $AX = T^* A$ has a solution. The following theorem due to Douglas \cite{Douglas1966} will be used in the sequel.

\begin{thm}[Douglas' Theorem]\label{t1.1}
    Let \( A, B \in \mathcal{B}(\mathcal{H}) \). Then the following conditions are equivalent:
    \begin{enumerate}
        \item \( \mathcal{R}(B) \subseteq \mathcal{R}(A) \).
        \item There exists \( \lambda > 0 \) such that \( B^* B \leq \lambda A A^* \).
        \item There exists \(C \in \mathcal{B}(\mathcal{H}) \) such that \( AC = B \).
    \end{enumerate}
\end{thm}

If any of these conditions hold, then there exists a unique operator \( W \in \mathcal{B}(\mathcal{H}) \) such that
\[
AW = B, \quad \mathcal{R}(W) \subseteq \overline{\mathcal{R}(A^*)}, \quad \mathcal{N}(W) = \mathcal{N}(B).
\]
The operator \( W \) is called the \textbf{reduced solution} of the equation \( AX = B \).

By the above observations we get that $T \in \mathcal{B}(\mathcal{H})$ admits an $A$-adjoint if and only if
 $\mathcal{R}(T^* A) \subseteq \mathcal{R}(A)$. In this paper, $\mathcal{B}_A(\mathcal{H})$ denotes the set of all $T \in \mathcal{B}(\mathcal{H})$ which admit an $A$-adjoint, that is

\[
\mathcal{B}_A(\mathcal{H}) = \{T \in \mathcal{B}(\mathcal{H}) : \mathcal{R}(T^* A) \subseteq \mathcal{R}(A) \}.
\]
The $\mathcal{B}_A(\mathcal{H})$ is a subalgebra of $\mathcal{B}(\mathcal{H})$ that is neither closed nor dense in $\mathcal{B}(\mathcal{H})$. It is easy to see that

\[
\overline{\mathcal{B}_A(\mathcal{H})} \subseteq \{T \in \mathcal{B}(\mathcal{H}) : \mathcal{R}(T^* A) \subseteq \overline{\mathcal{R}(A)} \}.
\]
By Theorem 1.1, we get that

\[
\mathcal{B}_{A^{\frac{1}{2}}}(\mathcal{H}) = \{T \in \mathcal{B}(\mathcal{H}) : \exists c > 0 \text{ such that } \|T\xi\|_A \leq c\|\xi\|_A, \quad \forall \xi \in \mathcal{H} \}.
\]

To fix notation, let \( (X, \mathcal{F}, \mu) \) be a \( \sigma \)-finite measure space, and let \( \mathcal{H} = L^2(\mu) \) be the Hilbert space of square-integrable complex-valued functions on \( X \). Let \( u: X \to \mathbb{C} \) be a measurable function. The multiplication operator \( M_u \) acts on \( L^2(\mu) \) by
\[
(M_u f)(x) = u(x) f(x), \quad f \in L^2(\mu).
\]
If \( u \geq 0 \) almost everywhere, then \( M_u \) is a positive operator. We define the modified inner product on \( \mathcal{H} \) as
\[
\langle f, g \rangle_A = \langle A f, g \rangle = \int_X u(x) f(x) \overline{g(x)} \, d\mu(x),
\]
which induces the norm
\[
\|f\|_A := \|A^{1/2} f\| = \left( \int_X u(x) |f(x)|^2 \, d\mu(x) \right)^{1/2}.
\]

Now, consider a measurable transformation $\varphi: X \to X$ that is \textbf{non-singular}, meaning that the pushforward measure $\mu \circ \varphi^{-1}$ is absolutely continuous with respect to $\mu$. Define the Radon-Nikodym derivative:
\begin{equation}
    h_\varphi = \frac{d(\mu \circ \varphi^{-1})}{d\mu},
\end{equation}
and let $E(f) = E(f \mid \varphi^{-1}(\mathcal{F}))$ denote the conditional expectation of $f$ with respect to the $\sigma$-algebra $\varphi^{-1}(\mathcal{F})$.

The \textbf{composition operator} associated with $\varphi$ is defined as
\begin{equation}
    (C_\varphi f)(x) = f(\varphi(x)), \quad \forall f \in L^2(\mu).
\end{equation}
The adjoint of $C_\varphi$ satisfies the following properties:

\begin{enumerate}
    \item {Adjoint formula:}
    \begin{equation*}
        C_\varphi^* f = h_\varphi E(f) \circ \varphi^{-1}.
    \end{equation*}
    \item {Product of the operator and its adjoint:}
    \begin{equation*}
        C_\varphi^* C_\varphi f = h_\varphi \cdot f.
    \end{equation*}
    \item {Composition with the adjoint operator:}
    \begin{equation*}
        C_\varphi C_\varphi^* f = (h_\varphi \circ \varphi) E f.
    \end{equation*}
    \item {Absolute value of the operator:}
    \begin{equation*}
        |C_\varphi| f = \sqrt{h_\varphi } \cdot f.
    \end{equation*}
\end{enumerate}

Now, let us define the \textbf{weighted composition operator} $W$ given by:
\begin{equation}
    (W f)(x) = u(x) f(\varphi(x)), \quad \forall f \in L^2(\mu).
\end{equation}
This operator combines both multiplication and composition, and it satisfies the following properties:

\begin{enumerate}
    \item {Adjoint formula:}
    \begin{equation*}
        W^* f = h_\varphi E(u \cdot f) \circ \varphi^{-1}.
    \end{equation*}
    \item {Composition of the operator and its adjoint:}
    \begin{equation*}
        W^* W f = h_\varphi E(u^2) \circ \varphi^{-1} \cdot f.
    \end{equation*}
    
    \begin{equation*}
        WW^* f = u \cdot (h_\varphi \circ \varphi) E(uf).
    \end{equation*}
    \item {Absolute value of the operator:}
    \begin{equation*}
        |W| f = \sqrt{h_\varphi E(u^2)\circ \varphi^{-1}}  \cdot f.
    \end{equation*}
\end{enumerate}
Although a significant body of work has addressed boundedness, compactness, and spectral properties of such operators (see \cite{Cowen1995, ManhasGhafri2024, Estaremi2024,  AlGhafri2020a, AlGhafri2021}), relatively few studies have examined their behavior within the semi-Hilbert space framework.

For any operator \( T \in \mathcal{B}_A(\mathcal{H}) \), the reduced solution to the equation
\[
A T^\sharp = T^* A
\]
is uniquely given by
\[
T^\sharp = A^\dagger T^* A,
\]
where \( A^\dagger \) denotes the Moore--Penrose inverse of \( A \) \cite{Douglas1966, Corach2008a}.

Our focus lies on generalized notions such as \( A \)-selfadjointness (\( T = T^\sharp \)), \( A \)-normality (\( T T^\sharp = T^\sharp T \)), \( A \)-quasinormality (\( T T^\sharp T = T^\sharp T^2 \)), and \( A \)-unitarity. These concepts extend classical operator theory and provide a deeper understanding of weighted dynamics under modified inner product structures. For more comprehensive studies on these properties, we refer the reader to \cite{AlMohammady2021,Arias2008, Benali2019,GuptaBhatia2015,HuangEstaremi2024,jab}. Recently, in \cite{ej} the authors have investigated some measure-theoretic characterizations for composition operators in some operator classes on $L^2(\Sigma )$-semi-Hilbertian spaces with respect to positive multiplication operators.\\

In this study, we establish a comprehensive framework for analyzing composition operators and their weighted variants in semi-Hilbertian settings induced by both multiplication and composition operators. By leveraging Douglas' range inclusion theorem, conditional expectations, and the Moore–Penrose inverse, we derive necessary and sufficient conditions for various operator-theoretic properties. The results not only generalize classical notions such as self-adjointness, normality, and isometry, but also offer new insights into the structural interplay between operators in weighted settings. Concrete examples are provided to demonstrate the effectiveness and sharpness of the developed criteria.

\section{Main Results}

In this section, we establish characterizations for when a composition operator \( C_{\varphi} \) on \( L^2(\mu) \) becomes \( A \)-selfadjoint, \( A \)-normal, \( A \)-quasinormal, \( A \)-isometric, or \( A \)-unitary, where \( A = M_u \) is a positive multiplication operator. We also present conditions under which \( C_{\varphi} \) acts as an \( A \)-partial isometry. These results are obtained using the notion of \( A \)-adjoint operators and the Douglas factorization theorem. Analogous results are also given when \( A \) is a positive composition operator. Each theorem is followed by proof and interpretation within the semi-Hilbert space context.

\begin{thm}
Let $M_u$ be a positive multiplication operator, $C_{\varphi}$ a composition operator, and define the weighted composition operator $W = M_u C_{\varphi}$ on $L^2(\mu)$. Then $C_{\varphi}$ is $M_u$-selfadjoint on $L^2(\mu)$ if and only if $M_u C_{\varphi} = C_{\varphi}^* M_u$, if and only if $W$ is selfadjoint (i.e., $W^* = W$).
\end{thm}

\begin{proof}
Recall that $C_{\varphi}$ is $M_u$-selfadjoint if and only if for all $f, g \in L^2(\mu)$, we have
\[
\langle C_{\varphi} f, g \rangle_{M_u} = \langle f, C_{\varphi} g \rangle_{M_u},
\]
which is equivalent to
\[
\langle M_u C_{\varphi} f, g \rangle = \langle M_u f, C_{\varphi} g \rangle.
\]
This, in turn, is equivalent to
\[
\langle W f, g \rangle = \langle M_u C_{\varphi} f, g \rangle = \langle C_{\varphi}^* M_u f, g \rangle = \langle W^* f, g \rangle.
\]
Hence we get that $C_{\varphi}$ is $M_u$-selfadjoint on $L^2(\mu)$ if and only if $W$ is selfadjoint.
\end{proof}
We now recall several equivalent conditions for the selfadjointness of a weighted composition operator \( W \) on the Hilbert space \( L^2(\mu) \).

\begin{thm}[\cite{Campbell1990}]\label{t1.3}
The weighted composition operator $W = M_u C_{\varphi}$ is selfadjoint on $L^2(\mu)$ if and only if the restriction $\varphi_{S_J} := \varphi|_{S_J}$ is periodic of period $2$, and 
\[
J := h_{\varphi} \cdot E(u) \circ \varphi^{-1} = u,
\]
where $S_J = \{ x \in X : J(x) \neq 0 \}$.
\end{thm}
 Combining the previous two theorems, we obtain the following corollary.

\begin{cor}\label{c1}
If $M_u$ is a positive multiplication operator, $C_{\varphi}$ is a composition operator and $W = M_u C_{\varphi}$ the related weighted composition operator on $L^2(\mu)$, then 
$C_{\varphi}$ is $M_u$-selfadjoint on $L^2(\mu)$ if and only if $\varphi_{S_J}$ is periodic of period $2$ and $J := h_{\varphi} E(u) \circ \varphi^{-1} = u$.
\end{cor} Let $T \in \mathcal{B}_A(\mathcal{H})$. Then there exists a unique operator associated with $T$, referred to as the $A$-adjoint, which solves the equation $AX = T^*A$ in the reduced sense. This operator, denoted by $T^\sharp$, is given by the formula $T^\sharp = A^\dagger T^* A$, where $A^\dagger$ is the Moore–Penrose inverse of $A$.

We recall that the operator $T\in\mathcal{B}_A(\mathcal{H})$ is called $A$-normal if $TT^{\sharp}=T^{\sharp}T$ and $T$ is $A$-quasi-normal if $TT^{\sharp}T=T^{\sharp}T^2$. Also, a bounded operator $T$ on the Hilbert space $\mathcal{H}$ is called hyponormal if $T^*T\geq TT^*$. If $T\in\mathcal{B}_A(\mathcal{H})$, then $T$ is called  $A$-hyponormal if $T^*AT\geq TAT^*$. Now in the next theorem we investigate normality and quasi-normality of the composition operator $C_{\varphi}$ with respect to the multiplication operator $M_u$ on the Hilbert space $L^2(\mu)$. 
\begin{thm}\label{t15}
Let \( M_u \) be a positive multiplication operator and \( C_{\varphi} \) a bounded composition operator on the Hilbert space \( L^2(\mu) \). Then the following statements hold:

\begin{itemize}
    \item[(a)] The operator \( C_{\varphi} \) is \( M_u \)-normal if and only if
    \[
    \left( \frac{1}{u} \chi_{S(u)} h \right) \circ \varphi \cdot u = \frac{1}{u} \chi_{S(u)} \cdot J,
    \]
    where \( \chi_{S(u)} \) denotes the characteristic function of the support \( S(u) = \{ x \in X : u(x) \neq 0 \} \), and \( J = h_{\varphi}  \cdot E(u) \circ \varphi^{-1} \).

    \item[(b)] The operator \( C_{\varphi} \) is \( M_u \)-quasi-normal if and only if
    \[
    \frac{1}{u} \chi_{S(u)} \cdot J \cdot h = \left( \frac{1}{u} \chi_{S(u)} \cdot J \right) \circ \varphi^{-1} \cdot h,
    \]
    where \( J = h_{\varphi}  \cdot E(u) \circ \varphi^{-1} \).
\end{itemize}
\end{thm}

\begin{proof}
(a) As we know, \( M_u^\dagger = M_{\frac{1}{u} \chi_{S(u)}} \) and \( C_\varphi^* = M_h C_{\varphi^{-1}} E \). Therefore,
\[
C_\varphi^\sharp = M_u^\dagger C_\varphi^* M_u = M_{\frac{1}{u} \chi_{S(u)}} M_h C_{\varphi^{-1}} E M_u.
\]
By definition, \( C_\varphi \) is \( M_u \)-normal if and only if \( C_\varphi^\sharp C_\varphi = C_\varphi C_\varphi^\sharp \). For every \( f \in L^2(\mu) \), we compute:
\[
C_\varphi^\sharp C_\varphi f = M_{\frac{1}{u} \chi_{S(u)}} M_h C_{\varphi^{-1}} E M_u C_\varphi f = \frac{1}{u} \chi_{S(u)} J \cdot f,
\]
\[
C_\varphi C_\varphi^\sharp f = C_\varphi M_{\frac{1}{u} \chi_{S(u)}} M_h C_{\varphi^{-1}} E M_u f = \left( \frac{1}{u} \chi_{S(u)} h \right) \circ \varphi \cdot E(uf).
\]
Hence,
\[
C_\varphi^\sharp C_\varphi f = C_\varphi C_\varphi^\sharp f \quad \text{if and only if} \quad \frac{1}{u} \chi_{S(u)} J \cdot f = \left( \frac{1}{u} \chi_{S(u)} h \right) \circ \varphi \cdot E(uf).
\]
This is equivalent to the integral equality
$$\int_X\frac{1}{u}\chi_{S(u)}J.fd\mu=\int_X(\frac{1}{u}\chi_{S(u)}h)\circ\varphi E(uf)d\mu=\int_X(\frac{1}{u}\chi_{S(u)}h)\circ\varphi ufd\mu,$$
for all \( f \in L^2(\mu) \). Therefore, \( C_\varphi \) is \( M_u \)-normal if and only if
\[
\left( \frac{1}{u} \chi_{S(u)} h \right) \circ \varphi \cdot u = \frac{1}{u} \chi_{S(u)} \cdot J.
\]

\medskip

(b) By definition, \( C_\varphi \) is \( M_u \)-quasi-normal if and only if
\[
C_\varphi^\sharp C_{\varphi^2} = C_\varphi C_\varphi^\sharp C_\varphi.
\]
For every \( f \in L^2(\mu) \), we compute:
\[
C_\varphi^\sharp C_{\varphi^2} f = M_{\frac{1}{u} \chi_{S(u)}} M_h C_{\varphi^{-1}} E M_u C_{\varphi^2} f = \frac{1}{u} \chi_{S(u)} J \cdot f \circ \varphi,
\]
\[
C_\varphi C_\varphi^\sharp C_\varphi f = C_\varphi M_{\frac{1}{u} \chi_{S(u)}} M_h C_{\varphi^{-1}} E M_u C_\varphi f = \left( \frac{1}{u} \chi_{S(u)} \cdot J \right) \circ \varphi \cdot f \circ \varphi.
\]
Thus,
\[
C_\varphi^\sharp C_{\varphi^2} f = C_\varphi C_\varphi^\sharp C_\varphi f
\]
if and only if
\[
\frac{1}{u} \chi_{S(u)} J \cdot f \circ \varphi = \left( \frac{1}{u} \chi_{S(u)} \cdot J \right) \circ \varphi \cdot f \circ \varphi.
\]
Equivalently, we obtain
\begin{align*}
\int_X \left( \frac{1}{u} \chi_{S(u)} J \right) \circ \varphi^{-1} \cdot h \cdot f \, d\mu 
&= \int_X \frac{1}{u} \chi_{S(u)} J \cdot f \circ \varphi \, d\mu \\
&= \int_X \left( \frac{1}{u} \chi_{S(u)} \cdot J \right) \circ \varphi \cdot f \circ \varphi \, d\mu \\
&= \int_X \frac{1}{u} \chi_{S(u)} \cdot J \cdot h \cdot f \, d\mu.
\end{align*}

for all \( f \in L^2(\mu) \). Therefore, \( C_\varphi \) is \( M_u \)-quasi-normal if and only if
\[
\frac{1}{u} \chi_{S(u)} J \cdot h = \left( \frac{1}{u} \chi_{S(u)} \cdot J \right) \circ \varphi^{-1} \cdot h.
\]
\end{proof}
We also recall that an operator \( T \in \mathcal{B}(\mathcal{H}) \) is called an \( A \)-isometry if 
\[
\|Tx\|_A = \|x\|_A, \quad \forall x \in \mathcal{H}.
\]
Hence, \( T \) is an \( A \)-isometry if and only if 
\[
\langle T^* A T x, x \rangle = \langle A T x, T x \rangle = \|T x\|_A^2 = \|x\|_A^2 = \langle A x, x \rangle, \quad \forall x \in \mathcal{H}.
\]
Therefore, \( T \in \mathcal{B}(\mathcal{H}) \) is an \( A \)-isometry if and only if
\[
T^* A T = A.
\]

Moreover, \( T \) is called \( A \)-unitary if both \( T \) and \( T^* \) are \( A \)-isometries. Thus, \( T \) is \( A \)-unitary if and only if
\[
T^* A T = T A T^* = A.
\]

In the following theorem, we investigate the conditions under which the composition operator \( C_{\varphi} \) is an isometry with respect to the multiplication operator \( M_u \) on the Hilbert space \( L^2(\mu) \).
\begin{thm}\label{t1.4}
Let \( M_u \) be a positive multiplication operator and \( C_{\varphi} \) a bounded composition operator on the Hilbert space \( L^2(\mu) \). Then \( C_{\varphi} \) is an \( M_u \)-isometry if and only if 
\[
u = J = h_{\varphi} \cdot E(u) \circ \varphi^{-1}.
\]
\end{thm}

\begin{proof}
The operator \( C_{\varphi} \) is an \( M_u \)-isometry if and only if 
\[
\|C_{\varphi} f\|_{M_u} = \|f\|_{M_u}, \quad \forall f \in L^2(\mu),
\]
which is equivalent to
\[
\|C_{\varphi} f\|_{M_u}^2 = \|f\|_{M_u}^2.
\]
By the definition of the \( M_u \)-inner product, this condition becomes
\[
\langle C_{\varphi}^* M_u C_{\varphi} f, f \rangle = \langle M_u f, f \rangle.
\]
It is known that 
\[
C_{\varphi}^* M_u C_{\varphi} f = h_{\varphi} \cdot E(u) \circ \varphi^{-1} \cdot f = J f.
\]
Therefore, \( C_{\varphi} \) is \( M_u \)-isometric if and only if \( J f = u f \) for all \( f \in L^2(\mu) \), that is,
\[
u = J = h_{\varphi} \cdot E(u) \circ \varphi^{-1}.
\]
\end{proof}
The operator \( T \in \mathcal{B}(\mathcal{H}) \) is called an \( A \)-partial isometry if 
\[
\|T x\|_A = \|x\|_A, \quad \forall x \in \mathcal{N}_A(T)^{\perp_A}.
\]
In the following, we aim to determine conditions under which the composition operator \( C_{\varphi} \) is an \( M_u \)-partial isometry on the Hilbert space \( L^2(\mu) \). As a first step, we compute the space \( \mathcal{N}_{M_u}(C_{\varphi})^{\perp_{M_u}} \). In the next lemma, we characterize this orthogonal complement as a subspace of \( L^2(\mu) \).

\begin{lem}\label{l1}
Let \( M_u \) be a positive multiplication operator and \( C_{\varphi} \) a bounded composition operator on the Hilbert space \( L^2(\mu) \). Then
\[
\mathcal{N}_{M_u}(C_{\varphi})^{\perp_{M_u}} = \left(u L^2(X \setminus S)\right)^{\perp},
\]
where \( J = h_{\varphi} \cdot E(u) \circ \varphi^{-1} \) and \( S = S(J) = \{ x \in X : J(x) \neq 0 \} \).
\end{lem}

\begin{proof}
A function \( f \in L^2(\mu) \) belongs to \( \mathcal{N}(u^{1/2} C_{\varphi}) \) if and only if
\[
\|u^{1/2} C_{\varphi} f\|^2 = \int_X u \cdot |f \circ \varphi|^2 \, d\mu = \int_X h_{\varphi} \cdot E(u) \circ \varphi^{-1} \cdot |f|^2 \, d\mu = \int_X J |f|^2 \, d\mu = 0.
\]
Thus, \( f \in \mathcal{N}(u^{1/2} C_{\varphi}) \) if and only if \( \|J f\| = 0 \), which implies \( f \in L^2(X \setminus S) \). Hence,
\[
\mathcal{N}_{M_u}(C_{\varphi}) = \mathcal{N}(u^{1/2} C_{\varphi}) = L^2(X \setminus S).
\]

Now, the \( M_u \)-orthogonal complement is given by:
\begin{align*}
\mathcal{N}_{M_u}(C_{\varphi})^{\perp_{M_u}} 
&= L^2(X \setminus S)^{\perp_{M_u}} \\
&= \left\{ f \in L^2(\mu) : \langle M_u g, f \rangle = \int_X u \cdot g \cdot \overline{f} \, d\mu = 0, \ \forall g \in L^2(X \setminus S) \right\}.
\end{align*}

This space is exactly the orthogonal complement (in the usual inner product) of \( u L^2(X \setminus S) \), i.e.,
\[
\mathcal{N}_{M_u}(C_{\varphi})^{\perp_{M_u}} = \left(u L^2(X \setminus S)\right)^{\perp}.
\]
\end{proof}
Now we are ready to provide necessary and sufficient conditions under which the composition operator \( C_{\varphi} \) is an \( M_u \)-partial isometry on the Hilbert space \( L^2(\mu) \).

\begin{thm}\label{t1.5}
Let \( M_u \) be a positive multiplication operator and \( C_{\varphi} \) a bounded composition operator on the Hilbert space \( L^2(\mu) \). Then \( C_{\varphi} \) is an \( M_u \)-partial isometry on \( L^2(\mu) \) if and only if
\[
\int_X (J - u) |f|^2 \, d\mu = 0, \quad \text{for all } f \in \left(u L^2(X \setminus S)\right)^{\perp},
\]
where \( J = h_{\varphi} \cdot E(u) \circ \varphi^{-1} \) and \( S = \{ x \in X : J(x) \neq 0 \} \).
\end{thm}

\begin{proof}
By Lemma~\ref{l1}, we know that
\[
\mathcal{N}_{M_u}(C_{\varphi})^{\perp_{M_u}} = \left(u L^2(X \setminus S)\right)^{\perp}.
\]
Therefore, \( C_{\varphi} \) is an \( M_u \)-partial isometry if and only if
\[
\|C_{\varphi} f\|_{M_u} = \|f\|_{M_u}, \quad \text{for all } f \in \left(u L^2(X \setminus S)\right)^{\perp}.
\]
This is equivalent to
\[
\|u^{1/2} f \circ \varphi\| = \|u^{1/2} f\|, \quad \forall f \in \left(u L^2(X \setminus S)\right)^{\perp}.
\]

Squaring both sides gives
\[
\int_X u \cdot |f \circ \varphi|^2 \, d\mu = \int_X u \cdot |f|^2 \, d\mu.
\]
Using the change-of-variable identity associated with the adjoint of the composition operator, we get:
\[
\int_X u \cdot |f \circ \varphi|^2 \, d\mu = \int_X J \cdot |f|^2 \, d\mu.
\]
Thus,
\[
\int_X (J - u) |f|^2 \, d\mu = 0, \quad \forall f \in \left(u L^2(X \setminus S)\right)^{\perp}.
\]
\end{proof}
In the next theorem, we investigate conditions under which the composition operator \( C_{\varphi} \) is \( M_u \)-unitary on the Hilbert space \( L^2(\mu) \). Specifically, we provide necessary and sufficient conditions for \( C_{\varphi} \) to be \( M_u \)-unitary.

\begin{thm}\label{t1.6}
Let \( M_u \) be a positive multiplication operator and \( C_{\varphi} \) be a bounded composition operator on the Hilbert space \( L^2(\mu) \). Then \( C_{\varphi} \) is \( M_u \)-unitary if and only if
\[
u = J = h_{\varphi} E(u) \circ \varphi^{-1} = (u \cdot h_{\varphi}) \circ \varphi.
\]
\end{thm}

\begin{proof}
By definition, \( C_{\varphi} \) is \( M_u \)-unitary if and only if
\[
C_{\varphi}^* M_u C_{\varphi} = M_u \quad \text{and} \quad C_{\varphi} M_u C_{\varphi}^* = M_u.
\]
For every \( f \in L^2(\mu) \), we compute:
\[
\langle C_{\varphi}^* M_u C_{\varphi} f, f \rangle = \langle h_{\varphi} E(u) \circ \varphi^{-1} f, f \rangle = \langle M_u f, f \rangle,
\]
and
\[
\langle C_{\varphi} M_u C_{\varphi}^* f, f \rangle = \langle (u \circ \varphi) \cdot (h_{\varphi} \circ \varphi) \cdot E(f), f \rangle = \langle M_u f, f \rangle.
\]
Therefore, \( C_{\varphi} \) is \( M_u \)-unitary if and only if
\[
u = h_{\varphi} E(u) \circ \varphi^{-1} \quad \text{and} \quad (u \cdot h_{\varphi}) \circ \varphi \cdot E(f) = u \cdot f \quad \text{for all } f \in L^2(\mu).
\]
To simplify the second condition, integrate both sides of the equation
\[
(u \cdot h_{\varphi}) \circ \varphi \cdot E(f) = u \cdot f
\]
and use the fact that \( E \) is a conditional expectation operator. It follows that the above holds for all \( f \in L^2(\mu) \) if and only if
\[
(u \cdot h_{\varphi}) \circ \varphi = u.
\]
Thus, the result follows.
\end{proof}

By Theorem~\ref{t1.6}, if \( C_{\varphi} \) is \( M_u \)-unitary, then
\[
R(M_u) = R\left(M_{(u \circ \varphi) \cdot (h_{\varphi} \circ \varphi)} E\right) = R\left(E M_{(u \circ \varphi) \cdot (h_{\varphi} \circ \varphi)}\right) \subseteq L^2(\varphi^{-1}(\Sigma)) = R(C_{\varphi}).
\]
Hence, by Theorem~\ref{t1.1}, there exists an operator \( C \in \mathcal{B}(L^2(\mu)) \) such that
\[
M_u C = C_{\varphi}.
\]

\begin{thm}\label{t1.9}
Let \( C_{\varphi} \) be a composition operator on \( L^2(\mu) \), with \( C_{\psi} \) positive. Then \( C_{\varphi} \) is \( C_{\psi} \)-selfadjoint if and only if the restricted map \( \varphi_{S_J} := \varphi|_{S_J} \) is periodic of period two, and
\[
J := h_{\varphi} \cdot E^{\varphi}(\sqrt{h_{\psi}}) \circ \varphi^{-1} = \sqrt{h_{\psi}},
\]
where \( S_J = \{x \in X : J(x) \neq 0\} \), and \( E^{\varphi}(\cdot) := E(\cdot \mid \varphi^{-1}(\mathcal{F})) \) denotes the conditional expectation with respect to the \(\sigma\)-algebra \( \varphi^{-1}(\mathcal{F}) \).
\end{thm}

\begin{proof}
Since \( C_{\psi} \) is a positive composition operator, it follows that
\[
C_{\psi} = |C_{\psi}| = M_{\sqrt{h_{\psi}}}, \quad \text{and} \quad C_{\psi}^{1/2} = M_{h_{\psi}^{1/4}}.
\]

The operator \( C_{\varphi} \) is \( C_{\psi} \)-selfadjoint if and only if
\[
C_{\psi} C_{\varphi} = C_{\varphi \circ \psi} = C_{\varphi}^* C_{\psi}.
\]
Since \( C_{\psi} = M_{\sqrt{h_{\psi}}} \), this condition becomes
\[
M_{\sqrt{h_{\psi}}} C_{\varphi} = C_{\varphi}^* M_{\sqrt{h_{\psi}}} = (M_{\sqrt{h_{\psi}}} C_{\varphi})^*.
\]

Define the weighted composition operator \( W_{\psi} := M_{\sqrt{h_{\psi}}} C_{\varphi} \). Then \( C_{\varphi} \) is \( C_{\psi} \)-selfadjoint if and only if \( W_{\psi} = W_{\psi}^* \), that is, \( W_{\psi} \) is selfadjoint.

Therefore, by Theorem~\ref{t1.3}, the operator \( W_{\psi} \) is selfadjoint if and only if \( \varphi_{S_J} = \varphi|_{S_J} \) is periodic of period two and
\[
J := h_{\varphi} \cdot E^{\varphi}(\sqrt{h_{\psi}}) \circ \varphi^{-1} = \sqrt{h_{\psi}}.
\]
This completes the proof.
\end{proof}
\begin{thm}\label{t1.10}
Let \( C_{\psi} \) be a positive composition operator and \( C_{\varphi} \) a bounded composition operator on the Hilbert space \( L^2(\mu) \). Let
\[
J := h_{\varphi} \cdot E^{\varphi}(\sqrt{h_{\psi}}) \circ \varphi^{-1}, \quad \text{and} \quad S := \{ x \in X : J(x) \neq 0 \}.
\]
Then the following statements hold:

\begin{enumerate}
    \item[(a)] The operator \( C_{\varphi} \) is a \( C_{\psi} \)-isometry if and only if
    \[
    \sqrt{h_{\psi}} = J = h_{\varphi} \cdot E^{\varphi}(\sqrt{h_{\psi}}) \circ \varphi^{-1}.
    \]

    \item[(b)] The operator \( C_{\varphi} \) is a \( C_{\psi} \)-partial isometry on \( L^2(\mu) \) if and only if
    \[
    \int_X (J - \sqrt{h_{\psi}}) |f|^2 \, d\mu = 0, \quad \text{for all } f \in \left( \sqrt{h_{\psi}} \cdot L^2(X \setminus S) \right)^{\perp}.
    \]

    \item[(c)] The operator \( C_{\varphi} \) is \( C_{\psi} \)-unitary if and only if
    \[
    \sqrt{h_{\psi}} = J = h_{\varphi} \cdot E^{\varphi}(\sqrt{h_{\psi}}) \circ \varphi^{-1}, \quad \text{and} \quad \left( \sqrt{h_{\psi}} \cdot h_{\varphi} \right) \circ \varphi = \sqrt{h_{\psi}},
    \]
    or equivalently,
    \[
    \left( \sqrt{h_{\psi}} \cdot h_{\varphi} \right) \circ \varphi \cdot E^{\varphi}(f) = \sqrt{h_{\psi}} \cdot f, \quad \forall f \in L^2(\mu).
    \]
\end{enumerate}
\end{thm}

\begin{proof}
The proof follows directly from the structure of the operator \( C_{\psi} = M_{\sqrt{h_{\psi}}} \), combined with the arguments used in Theorem~\ref{t1.9}. Applying the operator-theoretic characterizations from Theorems~\ref{t1.4}, \ref{t1.5}, and \ref{t1.6} to the weighted composition operator \( W_{\psi} = M_{\sqrt{h_{\psi}}} C_{\varphi} \), we obtain the stated equivalences.
\end{proof}
\begin{rem}\label{rem1.12}
If \( C_{\varphi} \) is \( C_{\psi} \)-unitary, then by Theorem~\ref{t1.10}, item (3), we have
\[
R(C_{\psi}) = R\left( M_{(\sqrt{h_{\psi}} \cdot h_{\varphi}) \circ \varphi} \cdot E^{\varphi} \right) = R\left( E^{\varphi} M_{(\sqrt{h_{\psi}} \cdot h_{\varphi}) \circ \varphi} \right).
\]
Therefore,
\[
L^2(\psi^{-1}(\Sigma)) = R(C_{\psi}) \subseteq R(C_{\varphi}) = L^2(\varphi^{-1}(\Sigma)).
\]
This implies that
\[
\psi^{-1}(\Sigma) \subseteq \varphi^{-1}(\Sigma).
\]
\end{rem}

As a direct consequence of this necessary condition, we obtain the following:

\begin{cor}
Let \( C_{\varphi} \) and \( C_{\psi} \) be composition operators on the Hilbert space \( L^2(\mu) \), with \( C_{\psi} \) positive. If
\[
\psi^{-1}(\Sigma) \nsubseteq \varphi^{-1}(\Sigma),
\]
then \( C_{\varphi} \) cannot be \( C_{\psi} \)-unitary.
\end{cor}

\section{Examples}

In this section, we present two concrete examples of non-singular transformations on the interval \( [0,1] \).  
These examples illustrate how the properties of \( M_u \)-normality, \( M_u \)-isometry, and \( M_u \)-unitarity of composition operators depend on the structure of the transformation and the choice of the weight function \( u \).

\begin{exam}
Let \( X = [0, 1] \), let \( d\mu = dx \), and let \( \Sigma \) be the Lebesgue \(\sigma\)-algebra. Define the non-singular transformation \( \varphi : X \to X \) by
\[
\varphi(x) =
\begin{cases}
2x, & \text{if } x \in [0, \tfrac{1}{2}], \\
2x - 1, & \text{if } x \in (\tfrac{1}{2}, 1].
\end{cases}
\]
Note that the Radon–Nikodym derivative is \( h(x) = \dfrac{d\mu \circ \varphi^{-1}}{d\mu} = 1 \).

For each \( 0 \leq a < b \leq 1 \) and \( f \in L^2(\Sigma) \), we have
\[
\int_{\varphi^{-1}(a,b)} f(x) \, dx = \int_{a/2}^{b/2} f(x) \, dx + \int_{(a+1)/2}^{(b+1)/2} f(x) \, dx = \int_{(a,b)} \frac{1}{2} \left\{ f\left(\tfrac{x}{2}\right) + f\left(\tfrac{1+x}{2}\right) \right\} dx.
\]
Hence,
\[
(E(f) \circ \varphi^{-1})(x) = \frac{1}{2} \left\{ f\left( \tfrac{x}{2} \right) + f\left( \tfrac{1+x}{2} \right) \right\}.
\]
It follows that
\[
E(f)(x) =
\begin{cases}
\frac{1}{2} \left\{ f(x) + f\left( \frac{1 + 2x}{2} \right) \right\}, & x \in [0, \tfrac{1}{2}], \\
\frac{1}{2} \left\{ f\left( \frac{2x - 1}{2} \right) + f(x) \right\}, & x \in (\tfrac{1}{2}, 1].
\end{cases}
\]

Let \( u : X \to \mathbb{R} \) be a positive, \( \Sigma \)-measurable function such that \( S(u) = X \). By Theorem~\ref{t15}, the operator \( C_{\varphi} \) is \( M_u \)-normal if and only if
\[
\frac{u(x)}{u(2x)} \chi_{[0,1/2]}(x) + \frac{u(x)}{u(2x - 1)} \chi_{(1/2,1]}(x) = \frac{u(x/2) + u((1+x)/2)}{2u(x)}, \quad \forall x \in [0,1].
\]
In particular, this identity implies that \( u(0) = u(1/2) \). Therefore, for any measurable function \( u \) such that \( u(0) \neq u(1/2) \), the operator \( C_{\varphi} \) is not \( M_u \)-normal. For example, if \( u(x) = e^x \), then \( C_{\varphi} \) is not \( M_u \)-normal.

\medskip

By Theorem~\ref{t1.4}, \( C_{\varphi} \) is \( M_u \)-isometric if and only if
\[
u(x) = \frac{u(x/2) + u((x+1)/2)}{2}, \quad \forall x \in [0,1].
\]

\medskip

Moreover, by Theorem~\ref{t1.6}, the operator \( C_{\varphi} \) is \( M_u \)-unitary if and only if
\[
u(x) = \frac{u(x/2) + u((x+1)/2)}{2} = u(2x)\chi_{[0,1/2]}(x) + u(2x - 1)\chi_{(1/2,1]}(x).
\]
\end{exam}

\begin{exam}
Let \( X = [0, 1] \), let \( d\mu = dx \), and let \( \Sigma \) be the Lebesgue \(\sigma\)-algebra. Define the non-singular transformation \( \varphi : X \rightarrow X \) by
\[
\varphi(x) = 
\begin{cases}
1 - 2x, & x \in [0, \tfrac{1}{2}], \\
2x - 1, & x \in (\tfrac{1}{2}, 1].
\end{cases}
\]
Note that the Radon–Nikodym derivative is \( h(x) = \dfrac{d\mu \circ \varphi^{-1}}{d\mu} = 1 \).

For each \( 0 \leq a < b \leq 1 \) and \( f \in L^2(\Sigma) \), we have
\begin{align*}
\int_{\varphi^{-1}(a, b)} f(x) \, dx 
&= \int_{\frac{1 - b}{2}}^{\frac{1 - a}{2}} f(x) \, dx + \int_{\frac{a + 1}{2}}^{\frac{b + 1}{2}} f(x) \, dx \\
&= \int_{(a, b)} \frac{1}{2} \left\{ f\left( \tfrac{1 - x}{2} \right) + f\left( \tfrac{1 + x}{2} \right) \right\} dx.
\end{align*}

Hence,
\[
(E(f) \circ \varphi^{-1})(x) = \frac{1}{2} \left\{ f\left( \tfrac{1 - x}{2} \right) + f\left( \tfrac{1 + x}{2} \right) \right\}.
\]

It follows that
\[
E(f)(x) = 
\begin{cases}
\frac{1}{2} \left\{ f(x) + f(1 - x) \right\}, & x \in [0, \tfrac{1}{2}], \\
\frac{1}{2} \left\{ f(-x) + f(x) \right\}, & x \in (\tfrac{1}{2}, 1].
\end{cases}
\]

Let \( u : X \to \mathbb{R} \) be a positive \( \Sigma \)-measurable function. Then, by Theorems~\ref{t15}, ~\ref{t1.4}   and~\ref{t1.6}, we have:

\begin{itemize}
    \item The operator \( C_{\varphi} \) is \( M_u \)-normal on the Hilbert space \( L^2(\Sigma) \) if and only if
   \begin{align*}
\frac{u(x)}{u(1 - 2x)} \chi_{[0, \frac{1}{2}]}(x) 
+ \frac{u(x)}{u(2x - 1)} \chi_{(\frac{1}{2}, 1]}(x) 
= \frac{u\left( \tfrac{1 - x}{2} \right) + u\left( \tfrac{1 + x}{2} \right)}{2u(x)}, 
\quad \forall x \in [0, 1].
\end{align*}

    \item The operator \( C_{\varphi} \) is \( M_u \)-isometric if and only if
    \[
    u(x) = \frac{u\left( \frac{1 - x}{2} \right) + u\left( \frac{1 + x}{2} \right)}{2}, \quad \forall x \in [0, 1].
    \]

    \item The operator \( C_{\varphi} \) is \( M_u \)-unitary if and only if
    \[
    u(x) = \frac{u\left( \frac{1 - x}{2} \right) + u\left( \frac{1 + x}{2} \right)}{2} = u(1 - 2x)\chi_{[0,1/2]}(x) + u(2x - 1)\chi_{(1/2,1]}(x), \quad \forall x \in [0, 1].
    \]
\end{itemize}
\end{exam}

%
%
%
%

\end{document}